\documentclass{article}

\usepackage{amsmath,amssymb,amsfonts,amsthm,latexsym,amstext,amscd}
\usepackage{epsfig,graphics,graphicx,color}
\usepackage{enumerate}
\usepackage{hyperref}

\newcommand{\bsx}{\boldsymbol{x}}%
\newcommand{\bsy}{\boldsymbol{y}}%
\newcommand{\bsalpha}{\boldsymbol{\alpha}}%

\newcommand{\NN}{\mathbb N}
\newcommand{\ZZ}{\mathbb Z}
\newcommand{\RR}{\mathbb R}

\newcommand{\PP}{\mathbb P}

\newcommand{\vol}{{\rm vol}}
\newcommand{\FF}{\mathbb{F}}

\newtheorem{theorem}{Theorem}[section]
\newtheorem{lemma}[theorem]{Lemma}

\newtheorem{example}[theorem]{Example}
\newtheorem{prop}[theorem]{Proposition}
\newtheorem{coro}[theorem]{Corollary}

\theoremstyle{definition}

\newtheorem{definition}[theorem]{Definition}

\numberwithin{equation}{section}

\begin{document}
\title{Variants of the Littlewood conjecture, their connection to uniformly distributed sequences, and the exact order of the discrepancy of van der Corput--Kronecker-type sequences}
\author       {Roswitha Hofer\thanks{Institute of Financial Mathematics and Applied Number Theory, Johannes Kepler University Linz, Altenbergerstr. 69, 4040 Linz, Austria. e-mail: roswitha.hofer@jku.at}}




\maketitle

\begin{abstract}The aims of this paper are twofold. First, it discusses the Littlewood conjecture and its variants with respect to uniformly distributed sequences. The second aim is 
to determine the exact order of the discrepancy of the van der Corput--Kronecker-type sequences which are based on recent counterexamples to the $X$-adic Littlewood conjecture over a finite field. Our result on the exact order of the discrepancy supports the well-established conjecture in the theory of uniform distribution, which states that $D_N\leq c \frac{\log^s N}{N}$, with $c>0$ for all $N>1$ is the best possible upper bound for the discrepancy $D_N$ of a sequence in $[0,1)^s$. 
\end{abstract}

%



\section{Uniform distribution, discrepancy and the Littlewood conjecture}\label{sec:1}

The Littlewood Conjecture and its variants have been the subject of frequent investigation in recent decades. Similar problems to these variants have arisen independently in the theory of uniform distribution, where the discrepancy of so-called Kronecker, Kronecker-type, Halton--Kronecker and Halton--Kronecker-type sequences has been investigated.
This manuscript aims to provide an overview of variants of the Littlewood conjecture and their relevance to specific uniformly distributed sequences. This overview is intentionally brief and does not provide a complete list of all references, but rather aims to offer a concise starting point. 

We start with the definition of the uniform distribution and the discrepancy of a sequence, before giving two well-established examples of one-dimensional uniformly distributed, low-discrepancy sequences.

\begin{definition}
	A sequence $(\bsx_n)_{n\geq 0}$ in $[0,1)^s$ is \emph{uniformly distributed} if 
	$$\lim_{N\to\infty}\frac{\#\{0\leq n<N: \bsx_n\in J\}}{N}\,=\,\vol(J)$$
	for every subinterval $J=\prod_{i=1}^s[a_i,b_i)$ of $[0,1)^s$ with $0\leq a_i<b_i\leq 1$ for $i=1,\ldots,s$. 

	The \emph{discrepancy} $D_N$ of the first $N\in\NN$ points of a sequence $(\bsx_n)_{n\geq 0}$ in $[0,1)^s$ is defined as  
	$$D_N(\bsx_n):=\sup_{J} \left|\frac{\#\{0\leq n<N: \bsx_n \in J\}}{N}-\vol(J)\right|$$
	where the supremum is extended over all such subintervals $J=\prod_{i=1}^s[a_i,b_i)$ of $[0,1)^s$. 
	The star-discrepancy $D^*_N$ is obtained when taking the supremum over subintervals of the form $\prod_{i=1}^s[0,b_i)$. 
\end{definition}
Note that for a sequence $(\bsx_n)_{n\geq 0}$ in $[0,1)^s$ we always have 
\begin{align}\label{eq:starvsnonstar}
D^*_N\leq D_N\leq 2^s D^*_N
\end{align}
 for every $N$. 
\begin{example}\label{ex:1}{\rm
	\begin{enumerate}
		\item Let $b\geq 2,\,b\in\NN$. We define the $b$-adic radical inverse function $\varphi_b(n)$ of $n\in\NN_0$ as $$\varphi_b(n)=\frac{n_0}{b}+\frac{n_1}{b^2}+\frac{n_2}{b^3}+\cdots,$$
		where $n=n_0+n_1b+n_2b^2+\ldots$ with $n_i\in\{0,1,\ldots,b-1\}$ is the unique base $b$ expansion of $n$. 
		The so-called $b$-adic van der Corput sequence $(\varphi_b(n))_{n\geq 0}$ is well known as {uniformly distributed} sequence and its {discrepancy} satisfies $ND_N(\varphi_b(n))=O(\log N)$. 
		\item Let $\alpha\in\RR$. Then the \emph{$(n\alpha)$-sequence}, $(\{n\alpha\})_{n\geq 0}$, is {uniformly distributed} if and only if $\alpha$ is irrational. 
		Represent $\alpha\in[0,1)$ in its simple continued fraction $\alpha=[a_1,a_2,a_3,\ldots]$ with convergents $p_j/q_j$ then the discrepancy of the \emph{$(n\alpha)$-sequence} satisfies 
		$$\Omega\left(\sum_{i=1}^{m(N)}a_i\right) =ND_N(\{n\alpha \})=O\left(\sum_{i=1}^{m(N)}a_i\right)$$ 
		where $q_{m(N)-1}<N\leq q_{m(N)}$ (cf. e.g. \cite[Corollary~1.64]{DT}). Consequently, if the coefficients in the simple continued fraction of $\alpha$ are bounded, then 
		\begin{equation}\label{eq:disc_nalpha}
		\Omega\left(\log N\right) =ND_N(\{n\alpha \})=O\left(\log N\right).
		\end{equation}
	\end{enumerate}
Here and elsewhere, the sumbol $\{\cdot\}$ denotes the fractional part function in $\RR$, which, for vectors in $\RR^n$, is applied componentwise. The symbol ``$=\Omega(f(N))$'' denotes ``$\geq c f(N)$ for infinitely many $N\in\NN$ with an implied constant $c>0$, which might depend on various parameters but which is independent of $N$''. ``$O(f(N))$'' denotes ``$\leq C f(N)$ for all $N\in\NN$ (or at least for all sufficiently large $N$) with an implied constant $C>0$ which may depend on various parameters, but which is independent of $N$.''} 
\end{example}

The bounds in \eqref{eq:disc_nalpha} are best possible in $N$, since a well-known general result of Schmidt \cite{schmidt} states that, for every one-dimensional sequence $(x_n)_{n\geq 0}$ in $[0,1)$ its discrepancy satisfies $ND_N(x_n)=\Omega(\log N)$, with an implied absolute constant $c>0$.


In the case of an $s$-dimensional sequence $(\bsx_n)_{n\geq 0}$ in $[0,1)^s$, the best-known lower bound for its discrepancy $(\bsx_n)_{n\geq 0}$ in $[0,1)^s$ is due to Bilyk, Lacey and Vagharshakyan \cite{BLV}, and is of the form 
$$ND_N(\bsx_n)=\Omega (\log^{s/2 +\eta_s} N)$$
 with an implied positive constant and with $0<\eta_s<1/2$, both depend on $s$, but are independent of $N$. This result improves upon and generalises earlier results by Roth \cite{roth}, Beck \cite{beck}, and Bilyk and Lacey \cite{BL}. 

In the theory of uniform distribution it is frequently conjectured, that apart from the implied constant, \begin{equation}\label{eq:ldseq}
ND_N(\bsx_n)=O( \log^s N)
\end{equation} is the best possible upper bound in $N$ that can be achieved for the discrepancy of an $s$-dimensional sequence $(\bsx_n)_{n\geq 0}$ in $[0,1)^s$. This is why sequences that satisfy a bound of the form given in equation \eqref{eq:ldseq} are called \emph{low-discrepancy sequences}. It is therefore of great interest to determine whether there is a lower bound for the discrepancy of a concrete example of an $s$-dimensional low-discrepancy sequence with $s > 1$ of the form $ND_N=\Omega(\log^s N)$, which then supports the conjecture that \eqref{eq:ldseq} is best possible. 
It should be emphasised that a counterexample to $ND_N=\Omega(\log^s N)$ in any dimension $s>1$ would raise many new questions and create numerous open problems in the theory of uniform distribution, as well as in the field of numerical integration algorithms known as quasi-Monte Carlo methods. (The interested reader is refered to e.g. \cite{niesiam} and \cite{DP} for more details on quasi-Monte Carlo algorithms). 
One of the most famous multi-dimensional low-discrepancy sequence is the so-called \emph{Halton sequence} in pairwise coprime bases $b_1,\ldots,b_s\geq 2,\, \in\NN$, i.e. $(\varphi_{b_1}(n),\ldots,\varphi_{b_s}(n))_{n\geq 0}$. This sequence is {uniformly distributed} and its discrepancy satisfies $$ND_N(\varphi_{b_1}(n),\ldots,\varphi_{b_s}(n))=O(\log^s N).$$
In \cite{LevinH1} Levin ensured the conjectured lower bound $$ND_N(\varphi_{b_1}(n),\ldots,\varphi_{b_s}(n))=\Omega(\log^s N)$$ for the discrepancy of the Halton sequence. This work was supplemented by further work of Levin \cite{LevinH2,Levints1,Levints2}, which ensures the conjectured lower discrepancy bound for various classes of low-discrepancy sequences. (See also \cite{faureLB} for a previous work and \cite{HoferPAMS} for a further low-discrepancy bound of the form $ND_N(\bsx_n)=\Omega (\log^{s} N)$ for a class of low-discrepancy sequences.) \\

Now let's return to the $(n \alpha)$-sequences. It is well known that the coefficients in the simple continued fraction of a real number $\alpha$ are bounded if and only if $$\alpha\in {\bf{Bad}}:=\{\beta \in\RR\,:\, \inf_{h\in\ZZ\setminus\{0\}} |h|\,\cdot\|h\beta\|=c_{\beta}>0\}.$$
Here $\|\cdot\|$ denotes the distance-to-the-nearest-integer function in $\RR$. 
Hence, if $\alpha\in  {\bf{Bad}}$ then \eqref{eq:disc_nalpha} holds for the discrepancy of the $(n \alpha)$-sequence. 

For a vector $\bsalpha=(\alpha_1,\alpha_2,\ldots,\alpha_s)\in\RR^s$ the so-called Kronecker sequence is defined as
	$(\{n\bsalpha\})_{n\geq 0}$, and serves as a multidimensional version of the $(n \alpha)$-sequence. It is known that a Kronecker sequence, $(\{n\bsalpha\})_{n\geq 0}$, is {uniformly distributed} if and only if $1,\alpha_1,\alpha_2,\ldots,\alpha_s$ are rationally independent. 

Although the Kronecker sequence has been studied extensively and intensively, the following two issues remain unresolved:

	\begin{itemize}
		\item is it true that for $s\geq 2$ we always have $ND_N(\{n(\alpha_1,\ldots,\alpha_s)\})=\Omega(\log^s N ) $.
		\item Is it possible to identify for $s\geq 2$ an example $(\alpha_1,\ldots,\alpha_s)\in\RR^s$ such that
		\begin{equation}\label{eq:Kron_LD}
		ND_N(\{n(\alpha_1,\ldots,\alpha_s)\})=O(\log^s N)  .
		\end{equation}
\end{itemize}

For the discrepancy of a Kronecker sequence, $(\{n\bsalpha\})_{n\geq 0}$, many proof techniques rely on lower bounds of
\begin{equation}\label{equ:1}
	|h_1\cdots h_s| \cdot\|h_1\alpha_1+\cdots+h_s\alpha_s\|
\end{equation}
depending on $(h_1,\ldots,h_s)\in(\ZZ\setminus\{0\})^s$. In short, one can say, the larger the value of \eqref{equ:1}, the smaller the obtained upper discrepancy bound.

We define a generalization of ${\bf{Bad}}$ namely the set ${\bf{Bad^s}}$ as {\small \begin{equation}\label{eq:LC1} \{(\alpha_1,\ldots,\alpha_s) \in\RR^s\,:\, \inf_{(h_1,\ldots,h_s)\in(\ZZ\setminus\{\boldsymbol{0}\})^s} |h_1\cdots h_s| \cdot\|h_1\alpha_1+\cdots+h_s\alpha_s\|=c_{\alpha_1,\ldots,\alpha_s}>0\}.\end{equation}}

The famous {Littlewood Conjecture}, which is unresolved for almost 100 years, states that for $s>1$ the set ${\bf{Bad^s}}$ is empty. In other word, for every $(\alpha_1,\ldots,\alpha_s) \in\RR^s$ we have
$$\inf_{(h_1,\ldots,h_s)\in(\ZZ\setminus\{\boldsymbol{0}\})^s} |h_1\cdots h_s| \cdot\|h_1\alpha_1+\cdots+h_s\alpha_s\|=0.$$

Although Littlewood's conjecture is widely accepted, the question of a multi-dimensional, low-discrepancy Kronecker sequence remains meaningful. Note that the existence of a vector $(\alpha_1,\ldots,\alpha_s)\in\RR^s$ such that \eqref{eq:Kron_LD} holds, is not equivalent to the existence of a vector $(\alpha_1,\ldots,\alpha_s)\in{\bf{Bad^s}}$. Example~\ref{ex:1}, Item~2 shows that this is even not the case if $s=1$.  

The Littlewood conjecture is usually found in literature in its most common form: for every $(\alpha_1,\alpha_2)\in\RR^2$ we have $$\inf_{h\geq 1} h\cdot\|h\alpha_1\|\cdot\|h\alpha_2\|=0.$$
Cassels and Swinnerton-Dyer \cite{Cas-S1955} proved the equivalence with its dual form: for every $(\alpha_1,\alpha_2)\in\RR^2$ we have $$\inf_{(h_1,h_2)\in\ZZ\times \ZZ\setminus\{(0,0)\}} \max\{|h_1|,1\}\max\{|h_2|,1\}\|h_1 \alpha_1+h_2\alpha_2\|=0$$

For weaker versions of \eqref{eq:LC1}, as
$$\inf_{(h_1,\ldots,h_s)\in(\ZZ\setminus\{\boldsymbol{0}\})^s} |h_1\cdots h_s|\cdot \varphi(|h_1\cdots h_s|) \cdot\|h_1\alpha_1+\cdots+h_s\alpha_s\|=c_{\alpha_1,\ldots,\alpha_s}>0$$
with a suitable function $\varphi$,
there are several results (see e.g. Schmidt \cite{schmidt1964}, Beck \cite{beck1994}, Bugeaud \& Moshchevitin \cite{BM2011}) with consequences for the discrepancy of the related Kronecker sequences. A result of Beck \cite[Theorem~1]{beck1994} states for example: 

\begin{theorem}\label{thm:beck}
	For almost all $(\alpha_1,\ldots,\alpha_s)\in\RR^s$, in the sense of Lebesgue measure, we have for every $\varepsilon>0$, 
	$$ND_N(\{n\bsalpha\})=O((\log N)^s (\log\log  N)^{1+\varepsilon})  $$
	with an implied constant depending on $(\alpha_1,\ldots,\alpha_s)$ and $\varepsilon$ but independent of $N$.
\end{theorem}

Finally the work of Einsiedler, Katok, and Lindenstrauss \cite{EKL06} should be definitely mentioned when discussing the Littlewood conjecture. They proved (see \cite[Theorem~1.5]{EKL06}) that the Hausdorff dimension of the set ${\bf{Bad^s}}$, i.e. the set of all mutually counterexamples to the Littlewood conjecture, is zero. The Littlewood conjecture was stated in the 1930s, after which many variants of it arose and were studied. The subsequent Subsections~\ref{subsec:1} and \ref{subsec:2} introduce variants that are particularly important for sequences in the theory of uniform distribution.

\subsection{The $p$-adic Littlewood conjecture}\label{subsec:1}

Let $p\in\PP$. The so-called \emph{$p$-adic Littlewood conjecture} (see Mathan and Teuli\'{e} \cite{MatTeu2004}) predicts for every $\alpha\in\RR$ that $$\inf_{h\in\ZZ\setminus\{0\},k\geq 0} |h|\,\cdot\,\|hp^k \alpha\|=0$$ or equivalently that the following set ${\bf{Bad}}_p$ is empty. 
$${\bf{Bad}}_p:=\{\alpha \in\RR\,:\, \inf_{h\in\ZZ\setminus\{0\},k\geq 0} |h|\,\cdot\,\|hp^k \alpha\|=c_{\alpha,p}>0\}.$$

Note that for a counterexample $\alpha$ to the $p$-adic Littlewood conjecture, i.e.  
 $\alpha\in{\bf{Bad}}_p$, the continued fraction coefficients of $\{p^k\alpha\}$ can be uniformly bounded for all $k\in\NN_0$. From \cite[Proposition~2]{HofLar2012} for such $\alpha\in {\bf{Bad}}_p$ it follows that the two-dimensional sequence $((\varphi_p(n),\{n\alpha\}))_{n\geq 0}$ would be a low-discrepancy sequence, i.e.
$$ND_N((\varphi_p(n),\{n\alpha\}))=O(\log^2 N). $$
It is still an open problem, whether the $p$-adic Littlewood conjecture holds true or there exists a $p\in\PP$ such that ${\bf{Bad}}_p\neq \emptyset$.

For weaker forms, as e.g.
$$\inf_{h\in\ZZ\setminus\{0\},k\geq 0} |h|\cdot\varphi(h,p^k) \cdot\|hp^k\alpha\|\geq c_{\alpha,p}>0$$
with a suitable function $\varphi$,
there are metrical results as well as results for e.g. algebraic $\alpha$ (see \cite{HofLar2012}, \cite{Lar2013}, \cite{Rid1958}, \cite{Bugetal2007}) with consequences for the discrepancy of the related sequence, $((\varphi_p(n),\{n\alpha\}))_{n\geq 0}$. See e.g. the following two theorems: 
\begin{theorem}[\cite{HofLar2012,Lar2013}] Let $p\in\PP$. For almost every $\alpha$ in $[0,1]$, in the sense of Lebesgue measure, we have 
	$$ND_N((\varphi_p(n),\{n\alpha\}))=O( \log^{2+\varepsilon} N) \quad \quad \mbox{for every $\varepsilon>0$},$$
	with an implied constant, which may depend on $\varepsilon$ and $\alpha$ but which is independent of $N$. 
\end{theorem}
\begin{theorem}[\cite{DHL2017}]  Let $p\in\PP$. For a real number $\alpha$ that is irrational as well as algebraic we have
	$$ND_N((\varphi_p(n),\{n\alpha\}))=O( N^{\varepsilon}) \quad \quad \mbox{for every $\varepsilon>0$,}$$
		with an implied constant, which may depend on $\varepsilon$ and $\alpha$ but which is independent of $N$.
\end{theorem}

See Einsiedler and Kleinbock \cite{EK2007} for a recent work on the $p$-adic Littlewood conjecture, which proves that the Hausdorff dimension of ${\bf{Bad}}_p$ is zero. For more recent work on putative counterexamples to the $p$-adic Littlewood conjecture, see \cite{Bla2024}. 

We would like to note that, according to \cite[Theorem~1.3]{EK2007}, the following variant of the Littlewood conjecture is true: Let $p_1,p_2$ be distinct prime numbers. Then for every $\alpha\in\RR$ we have 
$$\inf_{h\in\ZZ\setminus\{0\},r_1,r_2\geq 0}|h|\cdot\|hp_1^{r_1}p_2^{r_2}\alpha\|=0.$$
This excludes existence of an $\alpha\in\RR$ such that the continued fraction coefficients of $\{p_1^{r_1}p_2^{r_2}\alpha\}$ can be uniformly bounded for all $r_1,r_2\in\NN_0$. 
Consequently, there is no obvious candidate $\alpha\in\RR$ such that from \cite[Proposition~2]{HofLar2012} we know $ND_N((\varphi_{p_1}(n),\varphi_{p_2}(n),\{n\alpha\}))=O(\log^3 N)$. Nevertheless, \cite[Theorem~1.3]{EK2007} does not contradict existence of such an $\alpha$. 

\subsection{The Littlewood conjecture in the field of formal Laurentseries and its $X$-adic version}\label{subsec:2}

Replacing the set of real numbers with the set of formal Laurent series introduces further variants of the Littlewood conjecture. Before we state these variants, we will clarify the set of formal Laurent series and some relevant notions.

We write $\FF[X]$ for the ring of polynomials, $\FF(X)$ for the field of rational functions, and $$\FF((X^{-1}))=\left\{\sum_{i=j}^\infty a_iX^{-i}:j\in\ZZ, a_i \in\FF \mbox{ for all }i\geq j, a_j\neq 0\right\}$$ for the field of formal Laurentseries over a field $\FF$. 
For $\theta=\sum_{i=j}^\infty a_iX^{-i}\in\FF((X^{-1}))$ satisfying $a_j\neq 0$ we define as an extension of the degree of a polynomial the \emph{degree evaluation} $\deg(\theta)$ of $\theta$ as $\deg(\theta):=-j$,
	 the absolute value $|\theta|$ of $\theta$ as $|\theta|:=2^{-j}=2^{\deg(\theta)}$,
 the fractional part of $\theta$ as $\langle \theta\rangle:= \sum_{i=\max\{1,j\}}^\infty a_iX^{-i}$, and finally $\|\theta\|:=|\langle \theta\rangle|$ denotes the distance-to-the-nearest-polynomial. 

Davenport \& Lewis \cite{DavLew1963} asked whether or not
\begin{equation}\label{eq:FFLC}
\inf_{(Q_1,\ldots,Q_s)\in(\FF[X]\setminus\{\boldsymbol{0}\})^s} |Q_1\cdots Q_s| \cdot\|Q_1\theta_1+\cdots+Q_s\theta_s\|=0
\end{equation}
for every $(\theta_1,\ldots,\theta_s) \in\FF((X^{-1}))^s$ with $s>1$ in analogy to the Littlewood conjecture. 
In other words, one can say an element in $(\theta_1,\ldots,\theta_s) \in\FF((X^{-1}))^s$ with $s>1$ and
$$\inf_{(Q_1,\ldots,Q_s)\in(\FF[X]\setminus\{\boldsymbol{0}\})^s} |Q_1\cdots Q_s| \cdot\|Q_1\theta_1+\cdots+Q_s\theta_s\|=c_{\theta_1,\ldots,\theta_s}>0$$
would be a counterexample to this variant of the Littlewood conjecture \eqref{eq:FFLC}.
Davenport and Lewis \cite{DavLew1963} ensured the existence of a counterexample in the case where the field $\FF$ is infinite. Their work was complemented by Baker \cite{Baker1964} who provided an explicit counterexample in this case. An attempt by Armitage 1969 to construct a counter example over the finite field $\FF_p$ with $p>3$ failed as the proof was erroneous (see \cite{Arm1970}). Consequently, the function field analog of the Littlewood conjecture over a finite field remains an open problem. 

In particular, a counterexample over a finite field $\FF=\FF_p$ with $p\in\PP$, might be used to obtain low-discrepancy \textit{Kronecker-type sequences} in the sense of Larcher and Niederreiter \cite{LarNie1993}. 
For the construction of a Kronecker-type sequence first choose instead of $\alpha_1,\ldots,\alpha_s\in\RR$, formal Laurentseries $\theta_1,\ldots,\theta_s\in\FF_p((X^{-1}))$. Then for an index $n\in\NN_0$ define the polynomial $n(X)=n_0+n_1X+n_2X^2+\cdots$ via the base $p$ expansion of $n$, i.e. $n=n_0+n_1p+n_2p^2+\cdots$. Then compute 
$$\left<  \theta_i n(X)\right>\quad\mbox{ and finally set  $X=p$, which is denoted by ``$|_p$'',}$$
in order to obtain a number in $[0,1)$. 
The so constructed sequence $(\bsx_n)_{n\geq 0}$ in $[0,1)^s$ with 
$$\bsx_n=(\left<  \theta_1 n(X)\right>|_p,\ldots,\left<  \theta_s n(X)\right>|_p)$$ 
is then called a Kronecker-type sequence. Larcher and Niederreiter \cite[Theorem~2]{LarNie1993} identified counterexamples to the function field analog of the Littlewood conjecture over a finite field as potential candidates for constructing low-discrepancy Kronecker-type sequences. 
\begin{theorem} If $(\theta_1,\ldots,\theta_s)\in\FF_p((X^{-1}))^s$ is a counterexample to the Littlewood conjecture \eqref{eq:FFLC} then 
	the Kronecker-type sequence satisfies 
	$$ND_N((\left<  \theta_1 n(X)\right>|_p,\ldots,\left<  \theta_s n(X)\right>|_p))=O(\log^s N)$$
	with an implied constant only depending on $\theta_1,\ldots,\theta_s$, $p$, and $s$. 
\end{theorem}

Similar as for the original Littlewood conjecture weaker versions of \eqref{eq:FFLC} were studied and e.g. as a consequence the following result is known.  
\begin{theorem}[\cite{Lar1995}]
	For almost all $(\theta_1,\ldots,\theta_s)\in\FF_p((X^{-1}))^s$, in the sense of Haar-measure, we have for every $\varepsilon>0$: 
	$$ND_N((\left<  \theta_1 n(X)\right>|_p,\ldots,\left<  \theta_s n(X)\right>|_p))=O\big((\log N)^s (\log\log  N)^{2+\varepsilon}\big)  .$$
\end{theorem}
Note that there are unsolved problems in $\RR$ that have been solved over finite fields. One example is the function field analog of the famous Riemann hypothesis (see e.g. \cite{Fra2014}). 
Nevertheless, the following two items are still open for Kronecker-type sequences:

	\begin{itemize}
		\item Is it true that for $s\geq 2$ we always have $ND_N((\left<  \theta_1 n(X)\right>_p,\ldots,\left<  \theta_s n(X)\right>_p))=\Omega(\log^s N ) $?
		\item Is it possible to identify for $s\geq 2$ an example $(\theta_1,\ldots,\theta_s)\in\RR^s$ such that
		$$ND_N(\left<  \theta_1 n(X)\right>|_p,\ldots,\left<  \theta_s n(X)\right>|_p)=O(\log^s N)  ?$$
\end{itemize}

The last variant, which we would like to discuss in this manuscript, is the so-called $X$-adic Littlewood conjecture \cite{MatTeu2004}, the corresponding analog of the unsolved $p$-adic Littlewood conjecture. 
The $X$-adic Littlewood conjecture states for every $\theta \in \FF((X^{-1}))$, 
$$\inf_{r\in\NN_0,\,Q\in\FF[X]\setminus\{\boldsymbol{0}\}} |Q| \cdot\|QX^r\theta\|=0.$$
If the ground field $\FF$ is infinite counterexamples are already known, see \cite[Th\'{e}or\`eme~4.3]{MatTeu2004}. However, the method in \cite{MatTeu2004} fails for finite fields $\FF$. Note that a counterexample over a finite field will be interesting for the construction of a two-dimensional low-discrepancy \emph{van der Corput--Kronecker-type sequence}. For such a two-dimensional van der Corput--Kronecker-type sequence choose $p\in\PP$ and $\theta\in\FF_p((X^{-1}))$ and combine the $p$-adic van der Corput sequence $(\varphi_p(n))_{n\geq 0}$ with the Kronecker-type sequence $(\left<  \theta n(X)\right>|_p)_{n\geq 0}$. Levin \cite{levinLDhybrid} was the first who noted that a van der Corput--Kronecker-type sequence based on a counterexample to the $X$-adic Littlewood conjecture over $\FF_p$ is a low-discrepancy sequence. 
Robertson \cite{Robertson} recently worked out a detailed proof for the following Theorem~\ref{thm:upperbound}.

\begin{theorem}[\cite{Robertson}]\label{thm:upperbound}
	Let $\theta\in\FF_p((X^{-1}))$ be a counterexample to the $X$-adic Littlewood conjecture over $\FF_p$, i.e. $$\inf_{r\geq 0,Q\in\FF_p[X]\setminus\{\boldsymbol{0}\}} |Q| \cdot\|X^rQ\theta\|\geq c>0.$$ 
	Then the discrepancy of the two-dimensional van der Corput--Kronecker-type sequence satisfies $$N D_N((\varphi_p(n),{\left<  \theta n(X)\right>}|_p ))=O(\log^2 N)$$ with an implied absolute constant only depending on $\theta$ and $p$. 
\end{theorem}

Note that, similar to the other variants of the Littlewood conjecture, metrical results for weaker forms are known. See \cite{Robertson} and previously \cite{HoferdigKH}. In \cite{HoferdigKH} these metrical results were already applied to the discrepancy of so-called Halton--Kronecker-type sequences. In its most basic form such a sequence is a van der Corput--Kronecker-type sequence. For example in \cite{HoferdigKH} the following result was obtained.

\begin{theorem}
	Let $\FF=\FF_p$ then for almost all $\theta$, in the sense of Haar mesure, for every $\varepsilon>0$,
	$$ND_N((\varphi_p(n),{\left<  \theta n(X)\right>}|_p))=O(\log^{2+\varepsilon}(N)),$$
		with an implied constant only depending on $\theta$, $p$, and $\varepsilon$.
\end{theorem}

In contrast to the fact that \eqref{eq:FFLC} over a finite field $\FF=\FF_p$ is still open, a recent counterexample to the $X$-adic Littlewood conjecture over $\FF=\FF_3$ was discovered by Adiceam, Nesharim, and Lunnon \cite{Adiceametal}. This counterexample is defined as $\theta=\sum_{i\geq 1}f_iX^{-i}$ over $\FF_3$ with $(f_n)_{n\geq 1}$ is the paperfolding sequence in $\{0,1\}$, defined as 
$$f_n=\left\{ \begin{array}{cc} 0&k\equiv 1\pmod{4}\\
1& k\equiv 3\pmod{4}
	\end{array}\right.,$$
	with odd $k\in\NN$ such that $n=2^{\nu_2(n)}k$. 
Then $$\inf_{r\geq 0,Q\in\FF_3[X]\setminus\{\boldsymbol{0}\}} |Q| \cdot\|X^rQ\theta\|\geq 2^{-4}=c>0.$$ 

For larger $p=5,7,11$ further counterexample can be found in \cite{GR}. Note that the proofs in \cite{Adiceametal,GR} are computer assisted and therefore limited to small $p$. 
It remains an open problem to find explicit counter examples over every finite field $\FF_p$ with $p>2$. The restriction to positive $p\neq 2$ indicates that $p=2$ is different. In \cite[Conjecture~1.8]{GR} it's argued that the $X$-adic Littlewood conjecture might be true in the case where $\FF=\FF_2$.

With the specific counterexamples to the $X$-adic Littlewood conjecture over a finite field $\FF_p$, Theorem~\ref{thm:upperbound} gives new examples of low-discrepancy sequences.  
Having a new type of a low-discrepancy sequence in two dimensions, one is interested in a lower bound for its discrepancy. It should be emphasized once again at this point that for a concrete example of a low-discrepancy sequence in two-dimensions, a result of the form $N D_N=\Omega(\log^2 N)$ supports the well-established conjecture for the discrepancy while a smaller exact order of $ND_N$ in $N$ would bring a counterexample to this conjecture in uniform distribution, which would result in many new problems in the theory of uniform distribution. 
Our main theorem of this manuscript, Theorem~\ref{thm:lowerbound}, gives the desired lower bound for the discrepancy of these low-discrepancy sequences. 

\begin{theorem}\label{thm:lowerbound}
Let $\theta\in\FF_p((X^{-1}))$ be a counterexample to the $X$-adic Littlewood conjecture over $\FF_p$, i.e. $$\inf_{r\geq 0,Q\in\FF_p[X]\setminus\{\boldsymbol{0}\}} |Q| \cdot\|X^rQ\theta\|\geq c>0.$$  Then the discrepancy of the two-dimensional van der Corput--Kronecker-type sequence satisfies $$N D_N((\varphi_p(n),{\left<  \theta n(X)\right>}|_p))=\Omega(\log^2 N)$$ with an implied constant only depending on $\theta$ and $p$. 
\end{theorem}

It is a natural impulse to expand the two-dimensional sequence $$((\varphi_p(n),{\left<  \theta n(X)\right>}|_p))_{n\geq 0}$$ to a three-dimensional one with the aim of preserving the low-discrepancy property. 
Two different approaches come to mind here. The first is to define a three dimensional sequence $((\varphi_p(n),{\left<  \theta_1 n(X)\right>}|_p,{\left<  \theta_2 n(X)\right>}|_p))_{n\geq 0}$ with $\theta_1,\theta_2 \in\FF_p((X^{-1}))$ and study 
\begin{align}\label{equ:3dim1}
	\inf_{r\in\NN_0,\,Q_1,Q_2\in\FF[X]\setminus\{\boldsymbol{0}\}} |Q_1Q_2| \cdot\|(Q_1\theta_1+Q_2\theta_2)X^r\|.
\end{align}

Since it is open whether or not
$\inf_{Q_1,Q_2\in\FF[X]\setminus\{\boldsymbol{0}\}} |Q_1Q_2| \cdot\|(Q_1\theta_1+Q_2\theta_2)\|=0$
for every $\theta_1,\theta_2 \in\FF_p((X^{-1}))$, a three dimensional generalization of the form $((\varphi_p(n),{\left<  \theta_1 n(X)\right>}|_p,{\left<  \theta_2 n(X)\right>}|_p))_{n\geq 0}$ is not the first choice. 

As an alternative of \eqref{equ:3dim1} one can study whether or not
\begin{align}\label{equ:3dim2}\inf_{r_1,r_2\in\NN_0,\,Q\in\FF[X]\setminus\{\boldsymbol{0}\}} |Q| \cdot\|X^{r_1}(X+1)^{r_2}Q\theta\|=0\end{align}
for every $\theta\in\FF_p((X^{-1}))$. 
While counterexamples to \eqref{equ:3dim2} are known for infinite fields $\FF$, see \cite[Theorem~1]{BugMat2008}, a counterexample over a finite field $\FF_p$, i.e. $\theta\in\FF_p((X^{-1}))$, such that $\inf_{r_1,r_2\in\NN_0,\,Q\in\FF[X]\setminus\{\boldsymbol{0}\}} |Q| \cdot\|X^{r_1}(X+1)^{r_2}Q\theta\|\geq c>0$ is still an open problem. A counterexample over $\FF_p$ would bring similarly to Theorem~\ref{thm:upperbound} a result of the form 
$$ND_N((x_1,x_2,{\left<  \theta_1 n(X)\right>}|_p))=O(\log^3 N),$$
where $(x_1,x_2)_{n\geq 0}$ is a two-dimensional Halton-type sequence in the sense of the author. To avoid further quite technical definitions in this manuscript we refer the interested reader to  e.g. \cite{HoferdigKH} and \cite{Hofer} for more information on Halton-type sequences. However, it can be stated here that this is not the end of research in the area of variants of Littlewood's conjecture and the uniform distribution properties of related sequences, but rather that it will give rise to further interesting problems.\\

The rest of the paper is devoted to prove our main Theorem~\ref{thm:lowerbound} and reprove Theorem~\ref{thm:upperbound}. Section~\ref{sec:prerequ} collects the essential basic ideas of the concept of digital $(t,s)$-sequences in the sense of Niederreiter \cite{nie} as well as the most important properties of the simple continued fraction of an element in $\FF((X^{-1}))$. We identify a van der Corput--Kronecker-type sequence based on a counterexample $\theta$ of the $X$-adic Littlewood conjecture over $\FF_p$ as digital $(t,2)$-sequence, where $t$ will be quantified depending on the degrees of the simple continued fraction coefficients of $\langle X^r\theta\rangle$. Within the concept of digital $(t,2)$-sequences the low-discrepancy upper bound in Theorem~\ref{thm:upperbound} then immediately follows. Finally, in Section~\ref{sec:3} we develop the proof of the lower bound in Theorem~\ref{thm:lowerbound}, which is based on ideas from Levin \cite{Levints1}. 

\section{Identifying van der Corput--Kronecker-type sequences as digital sequences and a short proof of Theorem~\ref{thm:upperbound}}\label{sec:prerequ}

For some aspects of the class of van der Corput--Kronecker-type sequences, it is useful to interpret these sequences as digital $(t,s)$-sequences. In particular when investigating the discrepancy of a sequence the digital $(t,s)$-sequences concept, if applicable, is often the prefered method. We first define the digital method together with digital $(t,m,3)$-nets and digital $(t,2)$-sequences in the sense of Niederreiter (see e.g. \cite{niesiam}) as simple as possible to serve our purposes for Theorem~\ref{thm:upperbound} and Theorem~\ref{thm:lowerbound}. We do not distinguish between $\FF_p$ and the set $\{0,1,\ldots,p-1\}$.

\begin{definition}\label{def:t}
Let $p\in\PP$ and $m\in\NN$. Choose three $\NN\times m$-matrices $C_1,\,C_2,\,C_3$ over $\FF_p$, with $\NN$ indicating that they have infinitely many rows indexed by $1,2,3,\ldots$.  We construct a so-called digital pointset $(\bsx_n)_{0\leq n<p^m}$ of $p^m$ points as follows. Represent the nonnegative integer $n<p^m$ in base $p$, i.e.
$$ n = n_0+n_1p+\dots+n_{m-1}p^{m-1}\quad \mbox{with }0\leq n_j<p, $$
and set
$$ \vec n :=(n_0,\,\dots,\,n_{m-1})^T. $$
To generate the $i$th coordinate $x_n^{(i)}$ of $\bsx_n$, compute 
$$ C_i\cdot\vec n =:(y_1^{(i)},\,y_2^{(i)},\,\dots)^T\pmod{p}, $$
and set
$$ x_n^{(i)}:=\frac {y_1^{(i)}}p+\frac {y_2^{(i)}}{p^2}+\dots\,. $$
Let $t\in\NN_0$, $t<m$. We call the digital pointset $(\bsx_n)_{0\leq n<p^m}$ a \emph{digital $(t,m,3)$-net over $\FF_p$} if the following holds: 
for all $d_1,d_2,d_3\in\NN_0$ with $d_1+d_2+d_3\leq m-t$ the $(d_1+d_2+d_3)\times m$-matrix consisting of the

upper $d_1$ rows of $C_1$ together with the

upper $d_2$ rows of $C_2$ together with the

upper $d_3$ rows of $C_3$, 

\noindent has full row rank $d_1+d_2+d_3$.

For a digital sequence $(\bsx_n)_{n\geq 0}$ in $ [0,1)^2$ we choose two $\NN\times\NN_0$-matrices $C_1,\,C_2$ over $\FF_p$  with $\NN_0$ indicating that they have infinitely many columns indexed by $0,1,2,3,\ldots$. To generate the $i$th coordinate $x_n^{(i)}$ of $\bsx_n$, represent the integer $n\geq 0$ in base $p$, i.e.
$$ n = n_0+n_1p+\dots+n_rp^r\quad \mbox{with }0\leq n_j<p, $$
and with $n_r\neq 0$, set
$$ \vec n :=(n_0,\,\dots,\,n_r,\,0,\,0,\,\dots)^T, $$
and 
$$ C_i\cdot\vec n =:(y_1^{(i)},\,y_2^{(i)},\,\dots)^T\pmod{p}. $$
Further
$$ x_n^{(i)}:=\frac {y_1^{(i)}}p+\frac {y_2^{(i)}}{p^2}+\dots\,. $$
Let $t\in\NN_0$. We call the digital sequence $(\bsx_n)_{n\ge 0}$ a \emph{digital $(t,2)$-sequence over $\FF_p$} if the following holds: 
for every $m\in\NN, m>t$ and for all $d_1,d_2\in\NN_0$ with $d_1+d_2\leq m-t$ the $(d_1+d_2)\times m$-matrix consisting of the

left upper $d_1 \times m$-submatrix of $C_1$ together with the

left upper $d_2 \times m$-submatrix of $C_2$ 

\noindent has full row rank $d_1+d_2$.

\end{definition}

Usually the construction for digital sequences and digital nets is described for a general dimension $s$ and digital nets are based on $m\times m$ matrices instead of $\NN\times m$ matrices. For our proof of Theorem~\ref{thm:lowerbound} we need more exact digital information. 
It is well-known that digital $(t,s)$-sequences are low-discrepancy sequences (cf. e.g. \cite[Theorem~4.17]{niesiam}). See for instance \cite{faure,Hofer,HoNie,nie,NieYeo,sobol,NieXing} for different constructions of $(t,s)$-sequences and \cite{Levints1,Levints2} for lower bounds of the discrepancy of the form $ND_N=\Omega(\log^s N)$ for many of them.


The first step is to recognize the van der Corput--Kronecker-type sequences, which are based on counterexamples to the $X$-adic Littlewood conjecture over $\FF_p$, in Theorem~\ref{thm:upperbound} and in Theorem~\ref{thm:lowerbound} as digital $(t,2)$-sequences over $\FF_p$. For this, it is important to explain these sequences within the concept of the digital method. We start with Lemma~\ref{ref:identifymatrices} which recognizes van der Corput--Kronecker-type sequences as digital sequences. 

\begin{lemma}\label{ref:identifymatrices}
Let $p\in\PP$ and $\theta=\sum_{i=j}^\infty a_iX^{-i}\in\FF_p((X^{-1}))$. Let $I$ denote the $\NN\times \NN_0$ unit matrix over $\FF_p$. Let $H(\theta)$ be the \emph{Hankel matrix} determined by the Laurent series $\theta$ defined as
$$H(\theta)=\begin{pmatrix} a_1&a_2&a_3&\ldots \\
	a_2&a_3&a_4&\ldots \\
	a_3&a_4&a_5&\ldots \\
	\vdots &\ddots&\ddots&\ddots\end{pmatrix}\in\FF_p^{\NN\times \NN_0},$$
 with the setting $a_1=a_2=\cdots=a_{j-1}=0$ in the case where $j>1$. Then the van der Corput--Kronecker-type sequence $((\varphi_p(n),{\left<  \theta n(X)\right>}|_p))_{n\geq 0}$ corresponds with the two-dimensional digital sequence generated by $I$ and $H(\theta)$. 
\end{lemma}
\begin{proof}
	Note that 
	$$\langle \theta n(X)\rangle=\langle \left(\sum_{i=1}^\infty a_iX^{-i}\right) \left(\sum_{r=0}^\infty n_rX^r\right)\rangle=\sum_{j=1}^\infty\left(\sum_{l=0}^\infty a_{j+l}n_l\right)X^{-j}. $$
	The rest is straightforward.
	
\end{proof}

The next step is to quantify the $t$-value for van der Corput--Kronecker-type sequences which are based on counterexample to the $X$-adic Littlewood conjecture. Previous work related the $t$-value of one-dimensional Kronecker-type sequences to the degrees of the coefficients in the simple continued fraction expansion of $\theta$ (see e.g. \cite{LarNie1993} and \cite{HoferdigKH}). Corollary~\ref{coro:1} at the end of this section quantifies this $t$-value for van der Corput--Kronecker-type sequences, which are based on counterexample to the $X$-adic Littlewood conjecture, via a uniform bound for the degrees of the coefficients in the simple continued fraction expansions of $\langle X^r\theta\rangle,\,r\geq 0$. 
For the sake of self-consistency, the basic properties of the simple continued fraction expansion of a formal Laurentseries over $\FF_p$ are summarized. We denote the simple continued fraction expansion of $\theta$ by $\theta=[A_0;A_1,A_2,\ldots]$ with polynomials $A_0=\theta-\langle\theta \rangle$ and $A_i\in\FF_p(X),\,i\geq 1$ of degrees $\geq 1$. 
For $h\geq 1$ the $h$th convergent $P_h/Q_h$ of $\theta$ is a rational function in $\FF_p(X)$ defined by $[A_0;A_1,\ldots,A_h]=:P_h/Q_h$, with $P_h,Q_h\in\FF_p[X]$ satisfying $\gcd(P_h,Q_h)=1$. The degree $\deg(Q_h)$ of $Q_h$ is often abbreviated to $d_h$ and satisfies $d_h=\sum_{i=1}^h\deg(A_i)$. Furthermore, $\deg(\theta-P_h/Q_h)=-d_h-d_{h+1}$ or equivalently $\deg(\langle Q_h \theta\rangle)=-d_{h+1}$ for every $h\geq 1$. For all $k\in\FF_p[X]$ with $d_h\leq \deg(k)<d_{h+1}$ we have $$\deg(\theta-b/k)\geq \deg(\theta-P_h/Q_h) \mbox{ for all $b\in\FF_p[X]$}$$ or equivalently $\deg(\langle k\theta\rangle)\geq \deg(\langle Q_h\theta\rangle)=-d_{h+1}$. In order to refer to the different magnitudes, as $A_i,\,d_h,\,Q_h$ etc. for $\langle X^r \theta\rangle$ instead of $\theta$ we add a superscript $(r)$. For example $[A^{(r)}_1,A^{(r)}_2,\ldots]$ denotes the simple continued fraction expansion of $\langle X^r \theta\rangle$ and $d^{(r)}_h=\sum_{i=1}^h\deg(A^{(r)}_i)$. (For concise information on continued fractions and convergents in $\FF_p((X^{-1}))$ we refer the interested reader to e.g. the Appendix B of \cite{niesiam}.)

Lemma~\ref{prop:1} gives an important property about the rank structure of the Hankel matrix $H(\theta)$ determined by a Laurent series $\theta\in\FF_p((X^{-1}))$. This property depends on the continued fraction of $\theta$ more exactly on the degrees of the best approximation denomiators. 

\begin{lemma}\label{prop:1}
	Let $\theta=\sum_{k=j}^\infty a_kX^{-k}$ be any formal Laurent series over $\FF_p$. Let $m\in\NN$. Then
	the upper left $m\times m$ submatrix of $H(\theta)$, 
	$$\begin{pmatrix}
		a_1&a_2&\cdots&a_m\\
		a_2&a_3&\cdots&a_{m+1}\\
		\vdots&\vdots&&\vdots\\
		a_{m}&a_{m+1}&\cdots&a_{m+m-1} 
	\end{pmatrix} $$
	over $\FF_p$ is regular if and only if $m=d_h=\deg(Q_h)$ for an $h\in\NN$. 
\end{lemma}
\begin{proof}
	We observe for a linear combination of the columns of the matrix that 
	$$b_0\begin{pmatrix}
		a_1\\
		a_2\\
		\vdots\\
		a_{m}
	\end{pmatrix} +b_1\begin{pmatrix}
		a_2\\
		a_3\\
		\vdots\\
		a_{m+1}
	\end{pmatrix}+\cdots + b_{m-1}\begin{pmatrix}
		a_m\\
		a_{m+1}\\
		\vdots\\
		a_{m+m-1}
	\end{pmatrix} =\begin{pmatrix}
		0\\
		0\\
		\vdots\\
		0
	\end{pmatrix}$$
	with $b_0,b_1,\ldots,b_{m-1}\in\FF_p$ if and only if 
	$$\deg(\langle(b_0+b_1X+\cdots+b_{m-1}X^{m-1})\theta\rangle)< -m.$$
	
	Let $m=d_h$. We know $\deg(\langle Q_{h-1}\theta\rangle)=-d_{h}$ and for all $P\in\FF_p[X]\setminus\{0\}$ with $\deg(P)<d_{h}$ we have $\deg(\langle P\theta\rangle)\geq \deg(\langle Q_{h-1}\theta\rangle)=-d_{h}$. From the observation above we derive the linear independence of the columns for such $m$. 
	Now if $d_{h-1}<m<d_h$ use $P=Q_{h-1}$ with $\deg P<m$ but $\nu(\{P\theta\})=-d_h$. Hence the columns are linearly depending. 
\end{proof}

 The following Lemma~\ref{lem:Counterexample} quantifies the constant $c$ in $$\inf_{r\geq 0,Q\in\FF_p[X]\setminus\{\boldsymbol{0}\}} |Q| \cdot\|X^rQ\theta\|\geq c$$ via a uniform bound for the degrees of the coefficients in the simple continued fraction expansions of $\langle X^r\theta\rangle,\,r\geq 0$ and vice versa. 
\begin{lemma}\label{lem:Counterexample}
Let $\theta\in\FF_p((X^{-1}))$. The series $\theta$ is a counterexample to the $X$-adic Littlewood conjecture in $\FF_p((X^{-1}))$ if and only if the continued fraction expansion $[A^{(r)}_1,A^{(r)}_2,\ldots]$ of $\langle X^r\theta\rangle$ satisfies $\deg(A^{(r)}_j)\leq D(\theta)+1$ for every $j\geq 1$ and $r\geq 0$, where $D(\theta)\geq 0$ is a fixed constant in $\NN_0$. We say then $\theta$ is of finite deficiency $D(\theta)$, and for such $\theta$ we have $$\inf_{r\geq 0,Q\in\FF_p[X]\setminus\{\boldsymbol{0}\}} |Q| \cdot\|X^rQ\theta\|\geq 2^{-(D(\theta+1))}>0.$$ 
\end{lemma}
\begin{proof}
	Let $\theta$ be a counterexample to the $X$-adic Littlewood conjecture, i.e. $$\inf_{r\geq 0,Q\in\FF_p[X]\setminus\{\boldsymbol{0}\}} |Q| \cdot\|X^rQ\theta\|\geq c>0.$$
	We fix $r\in\NN_0$ and $Q=Q_h^{(r)}$ as a best approximation denominator of $\langle X^r \theta\rangle$. Then $\deg(\langle Q_h^{(r)}X^r \theta\rangle)=-d^{(r)}_{h+1}$, $\deg(Q_h^{(r)})=d^{(r)}_{h}$, and 
	$|Q| \cdot\|X^rQ\theta\|=2^{d^{(r)}_{h}-d^{(r)}_{h+1}}=2^{-\deg(A^{(r)}_{h+1})}\geq c$. Hence $\deg(A^{(r)}_{h}+1)\leq \log_2(1/c)$ for every $r\in\NN_0$ and $h\in\NN_0$. Conversely, if $\deg(A^{(r)}_j)\leq D(\theta)+1$ for every $j\in\NN$ and $r\in\NN_0$. Set $c:=2^{-(D(\theta)+1)}$. For any $r\geq 0$ and $Q\in\FF_p[X]\setminus\{\boldsymbol{0}\}$ choose the unique $h\in\NN_0$ such that $d^{(r)}_h\leq \deg(Q)<d^{(r)}_{h+1}$, then 
	$$\deg(Q)\deg(\langle QX^r \theta\rangle)\geq -d^{(r)}_{h+1}+d^{(r)}_h=-\deg(A^{(r)}_{h}+1)\geq -(D(\theta)+1).$$
	Thus 
	$$|Q|\cdot |\langle QX^r \theta\rangle|\geq 2^{-d^{(r)}_{h+1}+d^{(r)}_h}=2^{-\deg(A^{(r)}_{h}+1)}\geq 2^{-(D(\theta)+1)}=c.$$	
\end{proof}

\textbf{Proof of Theorem~\ref{thm:upperbound}. 
}
We use for the counterexample $\theta$ to the $X$-adic Littlwood conjecture the finite deficiency $D(\theta)$ due to Lemma~\ref{lem:Counterexample}. Let $m\in\NN$ satisfying $m>D(\theta)$. We make use of Lemma~\ref{ref:identifymatrices} and of Definition~\ref{def:t}. We start with $d_1,d_2\in\NN_0$ such that $d_1+d_2=m-D(\theta)$. We focus on the $(m-D(\theta))\times m$-matrix consisting of the

left upper $d_1 \times m$-submatrix of $C_1=I$ together with the

left upper $d_2 \times m$-submatrix of $C_2=H(\theta)$ 

\noindent and show its full row rank $m-D(\theta)$. In the following we set $r=d_1$. As $C_1$ is the unit matrix, the problem reduces to the upper left $m-d_1-D(\theta)\times m-d_1$ submatrix of the Hankel matrix $H(\langle X^{r}\theta\rangle)$.  Now choose $Q^{(r)}_{h-1}$ with $\deg(Q^{(r)}_{h-1})=d^{(r)}_{h-1}$ such that 
$d_2=m-d_1-D(\theta)\leq d^{(r)}_{h-1}\leq m-d_1$. Such a $d^{(r)}_{h-1}$ exists as the finite deficiency $D(\theta)$ of $\theta$ is assumed. Suppose $d^{(r)}_{h-1}< m-d_1-D(\theta)$ and $d^{(r)}_h> m-d_1$. Then $\deg(A^{(r)}_h)=d^{(r)}_h-d^{(r)}_{h-1}\geq m-d_1+1-(m-d_1-D(\theta)-1)=D(\theta)+2$ in contradiction to the finite deficiency $D(\theta)$. 
From Lemma~\ref{prop:1} we derive the regularity of the upper left $d^{(r)}_{h-1}\times d^{(r)}_{h-1}$ submatrix of $H(\langle X^{r}\theta\rangle)$. The first part of  Theorem~\ref{thm:upperbound} immediately follows from basic linear algebra. The low-discrepancy bound then follows e.g. from \cite[Theorem~4.17]{niesiam}. \qed

The above proof also covers the following corollary, which should definitely be noted here.  

\begin{coro}\label{coro:1}
	Let $\theta\in\FF_p((X^{-1}))$ be a counterexample to the $X$-adic Littlewood conjecture. Let $D(\theta)\in\NN_0$ be its finite deficiency due to Lemma~\ref{lem:Counterexample}. Then the matrices $I$ and $H(\theta)$ over $\FF_p$ are qualified to construct a digital $(D(\theta),2)$-sequence over $\FF_p$. 
	\end{coro}

\section{Proof of Theorem~\ref{thm:lowerbound}}\label{sec:3}

Following the ideas of Levin \cite{Levints1} we will derive a lower bound for the discrpancy of a suitable three dimensional net in order to obtain the lower bound in Theorem \ref{thm:lowerbound}. 
Let $m\in\NN$. We regard the three $\NN\times m$ matrices $I^{(m)},H^{(m)}(\theta)$ and $J^{(m)}$ over $\FF_p$. Here $I^{(m)}$ are the first $m$ columns of the $\NN \times \NN_0 $ unit matrix $I$, $H^{(m)}(\theta)$ consists of the first $m$ columns of $H(\theta)$, and $$J^{(m)}=\begin{pmatrix}
0&0&\cdots&0&1\\
0&0&\cdots&1&0\\
\vdots&\vdots&&\vdots&\vdots\\
0&1&\cdots&0&0 \\
1&0&\cdots&0&0 \\
0&0&\cdots&0&0\\
0&0&\cdots&0&0\\
\vdots&\vdots&&\vdots&\vdots
\end{pmatrix} $$ is the $\NN\times m$ \emph{upper antidiagonal matrix}. 
The following Lemma~\ref{lem:1} follows from Corollary~\ref{coro:1} using e.g. \cite[Lemma~4.38]{DP}. 
\begin{lemma}\label{lem:1}
Let $\theta\in\FF_p((X^{-1}))$ be a counterexample to the $X$-adic Littlewood conjecture over $\FF_p$ with finite deficiency $D(\theta)$ due to Lemma~\ref{lem:Counterexample}. Then for every $m>D(\theta)$ the three matrices $I^{(m)},H^{(m)}(\theta)$ and $J^{(m)}$ are qualified to construct a digital $(D(\theta),m,3)$-net over $\FF_p$. 
\end{lemma}

We will derive a lower bound for the star-discrepancy of this digital $(D(\theta),m,3)$-net which is stated in the subsequent Proposition~\ref{thm:3}. 

\begin{prop}\label{thm:3}
Let $\theta\in\FF_p((X^{-1}))$ be a counterexample to the $X$-adic Littlewood conjecture over $\FF_p$ with finite deficiency $D(\theta)$ due to Lemma~\ref{lem:Counterexample}. Let $m$ be a multiple of $8v$ with $v=3(D(\theta)+1)$, and large enough. Then the three-dimensional digital pointset $(\bsx_n)_{0\leq n<p^m}$ constructed by $I^{(m)},H^{(m)}(\theta)$ and $J^{(m)}$ satisfies 
$$p^mD^*_{p^m}(\bsx_n)\geq c\, m^2$$
with a positive constant $c$ depending on $\theta$ but independent of $m$. 
\end{prop}
Our Theorem~\ref{thm:lowerbound} follows then from Proposition~\ref{thm:3} together with the following Lemma~\ref{lem:2} (cf. \cite[Lemma 3.7]{niesiam}) and \eqref{eq:starvsnonstar}. 

Note that $J^{(m)}\vec{n}$ equals $(n_{m-1},\ldots,n_1,n_0,0\ldots)^T$ and therefore $x_n^{(3)}=\frac{n_{m-1}}{p}+\cdots+\frac{n_{1}}{p^{m-1}}+\frac{n_{0}}{p^m}=\frac{n}{p^m}$. 

\begin{lemma}\label{lem:2}
Let $S=(\bsx_n)_{n\geq 0}$ be an arbitrary sequence in $[0,1)^s$. Let $N=p^m$ and $P=(\bsy_n)_{n=0}^{p^m-1}$ with $\bsy_n:=(\bsx_n,n/p^m)\in[0,1)^{s+1}$. Then 
$$ND^*_N(\bsy_n)\leq \max_{1\leq M\leq N}MD^*_M(\bsx_n)+1.$$ 
\end{lemma}

Following ideas from Levin \cite{Levints1} we first study the so-called admissibility of the net. 
\begin{definition}\label{def:admissible}
We say a digital $(t,m,3)$-net over $\FF_p$ in the sense of Definition~\ref{def:t} is $d$-admissible if $$\min_{0\leq k<  n<p^m}\|\bsx_k\ominus \bsx_n\|_p>p^{-m-d}.$$
Here $\|\bsx_n\|_p=p^{-l}$ where $l=\sum_{i=1}^3\min\{j\in\NN\,:\,x_{n,j}^{(i)}\neq 0\}$ with $x_n^{(i)}=\sum_{j=1}^\infty x_{n,j}^{(i)}p^{-j}$. The minimum of the empty set is considered to be $\infty$. Further $\ominus$ denotes the componentwise digit-wise subtraction, i.e. $x^{(i)}_k\ominus x^{(i)}_n=:\sum_{j=1}^\infty \Big(x_{k,j}^{(i)}-x_{n,j}^{(i)}\pmod{p}\Big)p^{-j}$. 
\end{definition}
In words the admissibility guarantees a {digital minimum distance} between to distinct points of a digital $(t,m,3)$-net. The digital $(D(\theta),m,3)$-nets relevant in this paper satisfy such a digital minimum distance. 

\begin{lemma}\label{lem:admissible}
Let $\theta\in\FF_p((X^{-1}))$ be a counterexample to the $X$-adic Littlewood conjecture over $\FF_p$ with finite deficiency $D(\theta)$ due to Lemma~\ref{lem:Counterexample}. Then for every $m>D(\theta)$ the digital $(D(\theta),m,3)$-net over $\FF_p$ constructed by $I^{(m)},H^{(m)}(\theta)$ and $J^{(m)}$ is $D(\theta)+3$ admissible. 
\end{lemma}
\begin{proof}From the construction method in Definition \ref{def:t} it is easy to see that it suffices to prove $\min_{0 < n<p^m}\|\bsx_n\|_p>p^{-m-D(\theta)-3}$. 
Suppose there exists $n$ such that $0<n<p^m$ and $\|\bsx_n\|_p\leq p^{-m-D(\theta)-3}$. We write $\|x^{(i)}_n\|_p=:p^{-l_i-1}$, $i=1,2,3$. Note then our assumption means $l_1+l_2+l_3+3\geq m+D(\theta)+3$. 
As $n\neq 0$ we easily see $x_n^{(1)}\neq 0$, $x_n^{(2)}\neq 0$, and $x_n^{(3)}\neq 0$. For the second, note that $\theta$ as a counterexample is non rational. From the matrices $I^{({m})}$ and $J^{(m)}$ we see that $\vec{n}$ starts with $l_1$ zero entries and ends with $l_3$ zero entries and the in-between-part has to start and end with a nonzero entry and is of length at least one, i.e. $m-l_1-l_3\geq 1$. 
From $\|\bsx^{(2)}_n\|_p=:p^{-l_2-1}$ we derive some linear dependency in the upper $l_2\times m$ submatrix of $H^{(m)}(\theta)$ using the observation at the beginning of the proof of Lemma~\ref{prop:1}. More exactly when deleting the first $l_1$ and the last $l_3$ columns of it, the columns in the remaining $l_2\times m-(l_1+l_3)$ matrix are linearly depending. But $l_2\geq m-(l_1+l_3)+D(\theta)$. Similar as in the proof of Lemma~\ref{prop:1} together with the assumption that $\theta$ is of bounded deficiency $D(\theta)$ we achieve the desired contradiction. 
\end{proof}

\textbf{Proof of Proposition~\ref{thm:3}: }

We use the notation of Definition~\ref{def:t}. For proving Proposition~\ref{thm:3} assume $m$ large enough and $m$ to be a multiple of $8v$ with $v=3(D(\theta)+1)$. We will construct an interval $$J=[0,\gamma^{(1)})\times [0,\gamma^{(2)})\times [0,\gamma^{(3)})\subseteq [0,1)^3$$ such that 
\begin{equation}\label{eq:lb} 
\#\{0\leq n<p^m:\bsx_n\in J\}-p^m\lambda (J)\leq - c m^2
\end{equation}
with a fixed constant $c$ only depending on $\theta$ or $D(\theta)$ resp.

We write $\gamma^{(i)}=\sum_{j=1}^{r_i}\gamma_j^{(i)}p^{-j}$ in base $p$ and identify it with the vector $\vec \gamma^{(i)}=(\gamma^{(i)}_1,\ldots,\gamma^{(i)}_{r_i})^T$. 
We define $\vec{j}_v=(0,\ldots,0,1)$ of length $v$. And set 
$$\vec \gamma^{(1)}=(\underbrace{\vec{j}_v,\vec{j}_v,\ldots,\vec{j}_v}_{\mbox{$m/(4v)$ many.}})^T,\quad r_1=m/4,$$
$$\vec \gamma^{(2)}=(\underbrace{0,\ldots,0}_{\mbox{of length $v-u$.}},\underbrace{\vec{j}_v,\vec{j}_v,\ldots,\vec{j}_v}_{\mbox{$m/(4v)-1$ many.}})^T,\quad r_2=m/4-u$$
with $u$ satisfying $D(\theta)\leq u < 2D(\theta)$ and which will be chosen later. Further set 
$$\vec \gamma^{(3)}=(\underbrace{0,0,\ldots,0}_{\mbox{$m/4+u$ many.}},\underbrace{\ldots,\ldots,\ldots}_{\mbox{$m/4-u$ many fitting to $\gamma^{(2)}$.}},\underbrace{\vec{j}_v,\vec{j}_v,\ldots,\vec{j}_v}_{\mbox{$m/(4v)$ many.}})^T,\quad r_3=3m/4.$$ 
We introduce the truncation operator to $x=\sum_{j=1}^{\infty}x_jp^{-j}$ for $r\in\NN$ as $[x]_r:=\sum_{j=1}^{r}x_jp^{-j}$. 

First we ensure for a proper chosen $u$ existence of exactly one $n\in\{0,1,\ldots,p^m-1\}$ such that $[x_n^{(i)}]_{r_i}=\gamma^{(i)}$ for $i=1,2,3$. As $C_1=I^{(m)}$ we have 
$[x_n^{(1)}]_{r_1}=\gamma^{(1)}$ if and only if $(n_0,\ldots,n_{r_1-1})^T=\vec \gamma^{(1)}$. Thus the first $m/4$ entries of $\vec{n}$ are uniquely determined. 
The first $m/4+u$ zeroes in $\vec{\gamma}^{(3)}$ and the $m/4$ entries in the $(\vec{j}_v,\ldots,\vec{j}_v)$ string of $\vec{\gamma}^{(3)}$ uniquely determine the last $m/4+u$ entries of $\vec n$ as well as the $m/4$ entries of $\vec n$ indexed by $m/4,m/4+1,\ldots, m/2-1$. 
So the remaining $m/4-u$ entries of $\vec{n}$ should guarantee $[x^{(2)}_n]_{r_2}=\gamma^{(2)}$. We will show by fixing $u$ that there is exactly one choice of the remaining $m/4-u$ entries of $\vec{n}$ such that $[x^{(2)}_n]_{r_2}=\gamma^{(2)}$. Since the first $m/2$ digits of $n$ are fixed already, we reduce the solvability of $[x^{(2)}_n]_{r_2}=\gamma^{(2)}$ to the condition that $H^{(m/4-u)}(\langle X^{m/2}\theta\rangle)$ is regular. We choose $D(\theta)\leq u <2D(\theta)$ such that $H^{(m/4-u)}(\langle X^{m/2}\theta\rangle)$ is regular. Such a $u$ always exists under the assumption that $\theta$ has finite deficiency $D(\theta)$ together with Lemma~\ref{prop:1}. These now uniquely determined $m/4-u$ remaining entries of $\vec{n}$ will be used to define the middle part of $\gamma^{(3)}$. 
We summarize, there is a uniquely determined ${\overline{n}}\in\{0,1,\ldots,p^m-1\}$ such that $[x_{\overline{n}}^{i}]_{r_i}=\gamma^{(i)}$ for $i=1,2,3$ or equivalently $\bsx_{\overline{n}}\in[\gamma^{(1)},\gamma^{(1)}+1/p^{r_1})\times [\gamma^{(2)},\gamma^{(2)}+1/p^{r_2})\times [\gamma^{(3)},\gamma^{(3)}+1/p^{r_3})$. 

Now we split up $J=[0,\gamma^{(1)})\times [0,\gamma^{(2)})\times [0,\gamma^{(3)})$ into a union of disjoint elementary intervals as follows. We write $\gamma^{(i)}= \sum_{j=1}^{r_i}\gamma_j^{(i)}p^{-j}$ and $J$ as 
$$J=\prod_{i=1}^3 \bigcup_{j=1}^{r_i}\bigcup_{k=1}^{\gamma_j^{(i)}}\left[\left.[\gamma^{(i)}]_{j-1}+(k-1)p^{-j},[\gamma^{(i)}]_{j-1}+kp^{-j}\right)\right.$$
or equivalently
$$J=\bigcup_{j_1=1}^{r_1}\bigcup_{j_2=1}^{r_2}\bigcup_{j_3=1}^{r_3}\bigcup_{k_1=1}^{\gamma_{j_1}^{(1)}}\bigcup_{k_2=1}^{\gamma_{j_2}^{(2)}}\bigcup_{k_3=1}^{\gamma_{j_3}^{(3)}}\underbrace{\prod_{i=1}^3\left[\left.[\gamma^{(i)}]_{j_i-1}+(k_i-1)p^{-j_i},[\gamma^{(i)}]_{j_i-1}+k_ip^{-j_i}\right)\right.}_{=:I(j_1,k_1,j_2,k_2,j_3,k_3).}$$

Note that $I(j_1,k_1,j_2,k_2,j_3,k_3)$ is well-defined only if $\gamma_{j_1}^{(1)}\neq 0$, $\gamma_{j_2}^{(2)}\neq 0$, and $\gamma_{j_3}^{(3)}\neq 0$. The volume of $I(j_1,k_1,j_2,k_2,j_3,k_3)$ is then $p^{-(j_1+j_2+j_3)}$. We also say the elementary interval $I(j_1,k_1,j_2,k_2,j_3,k_3)$ has order $j_1+j_2+j_3$. 

Now we distinguish between the orders of $I(j_1,k_1,j_2,k_2,j_3,k_3)$ as follows. \\

If $j_1+j_2+j_3 \leq m-D(\theta)$ then $$\#\{0\leq n<p^m:\bsx_n\in I(j_1,k_1,j_2,k_2,j_3,k_3)\}-p^m\lambda (I(j_1,k_1,j_2,k_2,j_3,k_3))=0.$$
This follows from Definition~\ref{def:t} together with $d_1=j_1$, $d_2=j_2$, $d_3=j_3$, and basic linear algebra. \\

If $j_1+j_2+j_3 > m-D(\theta)$ we want to show that $$\#\{0\leq n<p^m:\bsx_n\in I(j_1,k_1,j_2,k_2,j_3,k_3)\}=0$$ or in other words $I(j_1,k_1,j_2,k_2,j_3,k_3)$ remains empty. Here we make use of Lemma~\ref{lem:admissible}. 

Note that always $j_1+j_2\leq m/2-u$. Thus for $j_1+j_2+j_3 > m-D(\theta)\geq m-u$ we need $j_3>m/2$. We observe $\gamma_{j_i}^{(i)}\neq 0$ for $i=1,2,3$ for $j_1,j_2,j_3$ satisfying $j_1+j_2+j_3 > m-D(\theta)\geq m-u$ if and only if 
$j_1=\lambda_1 v$, $j_2=v-u+\lambda_2 v$, and $j_3=m/2 +\lambda_3 v$ with $\lambda_1,\lambda_3\in\{1,\ldots,m/(4v)\}$, $\lambda_2\in \{1,\ldots,m/(4v)-1\}$ and $\lambda_1+\lambda_2+\lambda_3\geq m/(2v)$. Thus $j_1+j_2+j_3 \geq m+v-u$. 

Suppose there exists $\tilde{n}\in\{0,1,\ldots,p^m-1\}$ such that $\bsx_{\tilde{n}}\in I(j_1,k_1,j_2,k_2,j_3,k_3)$ for such $j_1,j_2,j_3$. Then 
$$\|\bsx_{\tilde{n}}\ominus \bsx_{\overline{n}}\|_p\leq p^{-(j_1+j_2+j_3)}\leq p^{-(m+v-u)}\mbox{ and }\overline{n}\neq \tilde{n}.$$

From Lemma~\ref{lem:admissible} and Definition~\ref{def:admissible} we know that 
$$\|\bsx_{\tilde{n}}\ominus \bsx_{{\overline{n}}}\|_p>p^{-(m+D(\theta)+3)}.$$ 
Since $v=3(D(\theta)+1)$ and $D(\theta)\leq u<2D(\theta)$ we see 
$$m+v-u>m+3D(\theta)+3-2D(\theta)=m+D(\theta)+3.$$
Therefore 
$$p^{-(m+D(\theta)+3)}>p^{-(m+v-u)}\geq \|\bsx_{\tilde{n}}\ominus \bsx_{{\overline{n}}}\|_p>p^{-(m+D(\theta)+3)}$$
yields the desired contradiction. Altogether

\begin{eqnarray*}
\#\{0\leq n<p^m:\bsx_n\in J\}-p^m\lambda (J)&\leq& -\sum_{j_1,j_2,j_3\geq m+v-u}p^{m-(j_1+j_2+j_3)}\\
&\leq& -\sum_{{j_1,j_2,j_3,}\atop {j_1+j_2+j_3= m+v-u}}p^{u-v}.
\end{eqnarray*}

It remains to ensure at least $c m^2$ with $c>0$ possible choices of such $(j_1,j_2,j_3)$. As $j_1=\lambda_1 v$, $j_2=v-u+\lambda_2 v$, and $j_3=m/2 +\lambda_3 v$ with $\lambda_1,\lambda_3\in\{1,\ldots,m/(4v)\}$, $\lambda_2\in \{1,\ldots,m/(4v)-1\}$ and $\lambda_1+\lambda_2+\lambda_3\geq m/(2v)$ we easily see that for each choice of $(\lambda_1,\lambda_2)\in \{m/(8v),\ldots,m/(4v)\}\times  \{m/(8v),\ldots,m/(4v)-1\}$ there is exactly one $\lambda_3\in \{1,\ldots,m/(4v)\}$ such that $\lambda_1+\lambda_2+\lambda_3=m/(2v)$. This completes the proof of Proposition~\ref{thm:3}. \qed






\end{document}